\documentclass[12pt]{amsart}
\numberwithin{equation}{section}
\usepackage{amsmath,amsthm,amsfonts,amscd,eucal}

\usepackage{xypic}
\usepackage{graphicx}
\usepackage{color}
\usepackage{amssymb}
\def\cb{{\mathcal B}}

\def\ce{{\mathcal E}}
\def\cf{{\mathcal F}}

\def\ch{{\mathcal H}}

\def\ck{{\mathcal K}}

\def\cam{{\mathcal M}}

\def\cs{{\mathcal S}}
\def\ct{{\mathcal T}}

\def\ga{{\mathfrak A}} 
\def\gb{{\mathfrak B}}

\def\gg{{\mathfrak G}}

\def\gi{{\mathfrak I}}

\def\gam{{\mathfrak M}}

\def\gar{{\mathfrak R}}

\def\gz{{\mathfrak Z}}

\def\bc{{\mathbb C}}
\def\bbf{{\mathbb F}}

\def\bn{{\mathbb N}}
\def\bp{{\mathbb P}}

\def\bz{{\mathbb Z}}

\def\a{\alpha}
\def\b{\beta}
\def\g{\gamma}  \def\G{\Gamma}
\def\d{\delta}

\def\l{\lambda} \def\L{\Lambda}

\def\m{\mu}

\def\r{\rho}

\def\s{\sigma} 

\def\f{\varphi}  \def\F{\Phi}
  
\def\om{\omega} \def\Om{\Omega}

\newtheorem{thm}{Theorem}[section]

\newtheorem{cor}[thm]{Corollary}
\newtheorem{prop}[thm]{Proposition}

\theoremstyle{definition}
\newtheorem{rem}[thm]{Remark}
\newtheorem{defin}[thm]{Definition}

\def\tr{\mathop{\rm Tr}}
\def\aut{\mathop{\rm Aut}}
\def\ad{\mathop{\rm ad}}

\def\spn{\mathop{\rm span}}

\def\ab{\mathop{\rm {\bf ab}}}
\newcommand{\ty}[1]{\mathop{\rm {#1}}}

\def\ad{\mathop{\rm ad}}

\def\idd{{1}\!\!{\rm I}}

\begin{document}

\baselineskip=17pt

\title[symmetries and ergodic properties]
{symmetries and ergodic properties in quantum probability}
\author{Vitonofrio Crismale}
\address{Vitonofrio Crismale\\
Dipartimento di Matematica\\
Universit\`{a} degli studi di Bari\\
Via E. Orabona, 4, 70125 Bari, Italy}
\email{\texttt{vitonofrio.crismale@uniba.it}}
\author{Francesco Fidaleo}
\address{Francesco Fidaleo\\
Dipartimento di Matematica\\
Universit\`{a} degli studi di Roma Tor Vergata\\
Via della Ricerca Scientifica 1, Roma 00133, Italy} \email{{\tt
fidaleo@mat.uniroma2.it}}


\begin{abstract}

\vskip0.1cm\noindent
We deal with the general structure of (noncommutative) stochastic processes by using the standard techniques of Operator Algebras.
Any stochastic process is associated to a state on
a universal object, i.e. the free product $C^*$-algebra in a natural way. In this setting one recovers the classical (i.e. commutative) probability scheme and many others, like those associated to the Monotone, Boolean and the $q$-deformed canonical commutation relations including the Bose/Fermi and Boltzmann cases. Natural symmetries like stationarity and exchangeability, as well as the ergodic properties of the stochastic processes are reviewed in detail for many interesting cases arising from Quantum Physics and Probability.
\end{abstract}


\maketitle

\section{introduction}
\label{sec1}

The concept of {\it Quantum Probability} has been introduced in the middle of seventies in the pioneering works of L. Accardi \cite{Acc, AFL}, R. L. Hudson \cite{CuHu,HP}, K. R. Parthasarathy \cite{HP}, and many other scientists. Since then, natural applications to various fields on mathematics and physics were carried out.
We mention the seminal investigation by D. V. Voiculescu involving the {\it Free Probability} and its applications to non hyperfinite type $\ty{II_{1}}$ von Neumann Factors \cite{Voi}, as well as the intersections with Harmonic Analysis firstly made by Bo\.{z}ejko \cite{Boz}. We also point out the remarkable connections, recently investigated in \cite{BCS, KS}, between quantum groups introduced by Woronowicz \cite{Wor} and Quantum Probability.

The present paper mainly deals with the investigation of the general structure of stochastic processes and their natural symmetries, like stationarity and exchangeability, by using the standard techniques of Operator Algebras, see \cite{CrF, CrFid,CFL}.
Although some of the pivotal results have been obtained in the above mentioned papers, our aim is to describe new ones and present all matter in a unified approach.
Thus the notes appear as an expository-research paper, since retracing the path to pursue our review purpose, it has given the chance to highlight some new properties for the first time here. We firstly show it is possible to view in an unified way the stochastic processes with sample space the unital, not necessarily commutative, $C^*$-algebra
$\ga$ and index set $J$. Indeed we show that the collection of these stochastic processes is in one-to-one correspondence with the state space of the free product $C^*$-algebra $*_J\ga$.
In particular, the subclasses of the exchangeable or stationary stochastic processes correspond to the convex compact (provided the processes are identity-preserving) subsets of the symmetric (i.e. invariant under finite permutations) or shift-invariant states, respectively. The {\it Algebraic Probability Space} $(\gb,\f)$ (see e.g. \cite{KS}) associated to the stochastic process described by the state $\om$ is then recovered by the Gelfand-Naimark-Segal (GNS for short) representation $(\ch_\om,\pi_\om,\Om_\om)$, as
$$
(\gb,\f):=\big(\overline{\pi_\om(*_J\ga)},\langle\pi_\om(\,{\bf\cdot}\,)\Om_\om\,\Om_\om\rangle\big)\,,
$$
where the closure is meant in norm or in the weak operator topology for the $C^*$ and $W^*$ cases, respectively. It provides the generalisation to the quantum case of the Kolmogorov Extension-Reconstruction Theorem.
It results that Quantum Probability, being considered a universal scheme, appears endowed with a deep degree of complexity in itself.

In the opinion of the authors, the main advantages of this unified description are the following.

First, it is possible to study all known cases directly on a suitable quotient algebra of a single one. Any quotient is obtained factoring out the free product $C^*$-algebra by the ideal generated by a "concrete" commutator. This allows to cover, e.g. the cases of  $q$-Canonical Commutation Relations, which includes the Bose/Fermi and the Boltzmann (i.e. free), or equally well the Boolean and the Monotone ones. The list is far from being complete, since one can add the commutative scheme which arises from the abelianisation of the free product.

Second, the natural symmetries of the stochastic processes like stationarity and exchangeability, as well as their ergodic properties can be managed using the standard results of Ergodic Theory, see e.g. Chapter 4 of \cite{BR1}.

As an example, we mention the equivalence between some factorisation rules naturally emerging in Quantum Statistical Mechanics and the property of the convergence to equilibrium for stationary and symmetric states (cf. Theorem \ref{pscm}). The features above are the content of Section \ref{sec2}.

The following sections of the paper are instead devoted to review the applications of the general results of Section \ref{sec2} to the ergodic properties of the stochastic processes arising from the genuine quantum cases. The main results are complemented by the sketch of their proofs for the convenience of the reader.

In Section \ref{sec3}, we first connect some ergodic/clustering properties of a given stochastic process with some algebraic properties of its corresponding state, that is the {\it product state} or the {\it block singleton} conditions (see Definition \ref{bspsc}). This leads to clarify here the role played by the block singleton condition in Quantum Probability  as the right noncommutative analogue of the product state condition. Indeed, under the invariance conditions of stationarity or exchangeability, it is shown that the states realising the block singleton condition are exactly those satisfying the noncommutative analogue of the convergence to the equilibrium \eqref{eq11}. Moreover the product state condition results equivalent to the ergodic property of weak clustering \eqref{eq1}. Such results find a natural application to the so called Haagerup states \cite{Ha} on the group $C^*$-algebra of the free group on infinitely many generators, which appear in Free Probability.

Section \ref{sec4} is devoted to examples. We review in a self-containing form de Finetti-type results and ergodic properties for stationary and symmetric states in some concrete $C^*$-algebras, plenty of them coming from physical investigations. The cases of $q$-deformed, $-1<q<1$ \cite{BSK, CrFid}, Bose \cite{CFL, St2}, Fermi \cite{CrF, CrFid}, Boolean \cite{Boz, CrFid, CFL} and Monotone \cite{CFL, Lu} processes are described in detail.

\section{exchangeable and stationary stochastic processes}
\label{sec2}

Fix a $C^*$-algebra $\ga$ and an index set $J$. We suppose without further mention that $\ga$ is unital and all morphisms preserve the algebraic structure, including the $*$ operation.
For simplicity we may think of $J=\bz$ to achieve the two sided shift, but general index sets are allowed as well as non unital sample algebras $\ga$, or non identity preserving stochastic processes.

A dynamical system based on the group $G$ is a pair $(\ga,\a)$, where $\ga$ is a $C^*$-algebra and $\a:g\in G\mapsto\a_g\in\aut(\ga)$ is an action on it.
To achieve dissipative dynamics, one needs to consider merely completely positive linear maps $\a_g$. This is the case when simply a monoid naturally acts on $\ga$. Nevertheless in the present paper we consider only dynamical systems based on automorphisms.
The {\it fixed point subalgebra} is defined as $\ga^G:=\{a\in\ga\mid \a_g(a)=a\,,g\in G\}$.
We denote the $*$-weakly compact convex subset of the $G$-invariant states by $\cs_G(\ga)$, and the collection of its extremal (i.e. ergodic) states by $\ce(\cs_G(\ga))$. For $(\ga, \a)$ as above and an invariant state $\f$ on $\ga$,
$(\pi_\f,\ch_\f,U_\f,\Omega_\f)$ is the GNS covariant
quadruple canonically associated to $\f$, see e.g. \cite{BR1}.
As usual, $\gz_\f:=\pi_\f(\ga)''\bigwedge\pi_\f(\ga)'$ is the centre of $\pi_\f(\ga)''$.

In these notes we mainly deal with a couple of groups. Namely we take the group $\bp_J:=\bigcup\{\bp_I|I\subseteq J\, \text{finite}\}$, given by the permutations on $J$ leaving fixed all of its elements but a finite number of them, or we get $G=\bz$. In the latter case the action of $G$ is generated by a single automorphism (i.e. the shift) $\a\in\aut(\ga)$. A state is called is called {\it symmetric} or {\it shift-invariant} if it belongs to $\cs_{\bp_J}(\ga)$ or $\cs_{\mathbb{Z}}(\ga)$, respectively.

Consider the unital free product $C^{*}$-algebra $*_\bz\ga$ based on a single $C^{*}$-algebra $\ga$, see e.g. \cite{VDN}. For $j\in\bz$, denote
$i_j:\ga\to *_\bz\ga$ the canonical injections of $\ga$ into $*_\bz\ga$. Then both $\bp_\bz$ and $\bz$ are naturally acting on $*_\bz\ga$ by considering the permutations and the shift of the indices.
\begin{prop}
\label{symsh}
We have $\cs_{\bp_\bz}(*_\bz\ga)\subset\cs_{\bz}(*_\bz\ga)$.
\end{prop}
\begin{proof}
For $\{j_1,j_2\dots, j_n\}\subset\bz$ with possibly repeated indices such that the contiguous ones are different (i.e. $j_i\neq j_{i+1}\,, i=1,\dots,n-1$),
the elements $X:=i_{j_1}(A_1)i_{j_2}(A_2)\cdots i_{j_n}(A_n)$
with $A_j\in\ga$, generate $*_\bz\ga$. By a standard approximation argument, we reduce the matter to such generators.
For $X$ and the corresponding sequence of indices as above, there exists a finite interval $J_X=[k,l]\subset\bz$ with $\{j_1,j_2\dots, j_n\}\subset J_X$.
In addition, there exists a cycle
$\g_X\in\bp_\bz$ such that $[k+1,l+1]=\g_X(J_X)$. For $\f\in\cs_{\bp_\bz}(*_\bz\ga)$, after denoting by $\a$ and $\a_g$ the one step shift and the action of $\bp_\bz$, respectively, we get
$$
\f(\a(X))=\f(\a_{\g_X}(X))=\f(X)\,.
$$
\end{proof}
One of the fundamental achievements in the theory of stochastic processes (classical or not) allows to find some sufficient conditions to construct a process starting from the knowledge of a collection of finite dimensional distributions. In the abelian case they are summarised in the Kolmogorov Reconstruction Theorem, whereas the quantum generalisation is provided by the GNS construction.

Fix $n\in\bn$, $\{j_1,j_2\dots, j_n\}\subset J$ with contiguous different indices,
and elements $\{A_1,A_2\dots, A_n\}\subset \ga$. The finite joint distributions are the values $p_{j_1,j_2\dots, j_n}(A_1,A_2\dots, A_n)$ which arise from multilinear functionals
$\{p_{j_1,j_2\dots, j_n}\}_{j_1,j_2\dots, j_n\in J}$ on $\ga$. They satisfy some natural natural positivity and consistency conditions given by
\begin{align*}
(i)\,\,\,\,\,\,&p_{j_n,\dots,j_2,j_1,j_2,\dots,j_n}(A_n^*,\dots,A_1^*A_1,\dots,A_n)\geq0\,\,  (positivity)\\
(ii)\,\,\,\,\,&p_{j_1,\dots, j_{k-1},j_{k},j_{k-1},\dots, j_{n}}(A_1,\dots,A_{k-1},\idd,A_{k+1},\dots,A_n)\\
=&p_{j_1,\dots, j_{k-1},j_{k+1},\dots, j_n}(A_1,\dots,A_{k-1},A_{k+1},\dots, A_n)\,\, (consistency).
\end{align*}
In the classical case, i.e. $\ga=C(I)$, the algebra of the continuous functions on the compact space $I$, the above properties reduces to the Kolmogorov requests. Thus one can construct a probability measure $\m$ on the Tikhonoff product $\prod_J I$ of $J$ copies of $I$. In the quantum setting they allow to perform the GNS representation (defined up to unitary equivalence) and so give rise to general stochastic processes, as defined in the forthcoming lines. In order to avoid technicalities, we assume as starting point (i.e. by definition) that the process under consideration is directly realised on a Hilbert space, corresponding to $L^2\big(\prod_J I,\m\big)$ in the classical situation.
\begin{defin}
\label{upun}
A (realisation of the) {\it stochastic process} labelled by the index set $J$ is a quadruple
$\big(\ga,\ch,\{\iota_j\}_{j\in J},\Om\big)$, where $\ga$ is a $C^{*}$-algebra, $\ch$ is an Hilbert space,
the $\iota_j$'s are $*$-homomorphisms of $\ga$ in $\cb(\ch)$, and
$\Om\in\ch$ is a unit vector, cyclic for  the von Neumann algebra
$M:=\bigvee_{j\in J}\iota_j(\ga)$ naturally acting on $\ch$. The process is said to be {\it unital} if $
\iota_j(\idd_\ga)=I_\ch$, $j\in J$.

The process is said to be {\it exchangeable}, or {\it stationary} whenever $J$ is $\bz$ or $\bn$, if for any $n\in\bn$, $j_1,\ldots j_n\in J$, $A_1,\ldots A_n\in\ga$
$$
\langle\iota_{j_1}(A_1)\cdots\iota_{j_n}(A_n)\Om,\Om\rangle
=\langle\iota_{g(j_{1})}(A_1)\cdots\iota_{g(j_{n})}(A_n)\Om,\Om\rangle\,,
$$
for $g\in\bp_J$, or $g(j_{l})=j_{l}+1$, respectively.
\end{defin}
\vskip0.2cm
\noindent
It is easy to see that the quadruple $\big(\ga,\ch,\{\iota_j\}_{j\in J},\Om\big)$ uniquely realises, up to unitary equivalence, the stochastic process by which we mean that it provides all its joint finite distributions:
\begin{equation*}
p_{j_1,j_2\dots, j_n}(A_1,A_2\dots, A_n):=\langle\iota_{j_1}(A_1)\cdots\iota_{j_n}(A_n)\Om,\Om\rangle\,.
\end{equation*}
From now on, if not otherwise specified, we only deal with unital stochastic processes.

Consider the unital free product $C^{*}$-algebra $*_J\ga$.
One can see that the a stochastic process uniquely defines a state $\f\in\cs(*_J\ga)$ and viceversa.
\begin{thm}
\label{mmin}
The unitary equivalence class determined by the quadruple $\big(\ga,\ch,\{\iota_j\}_{j\in J},\Om\big)$ uniquely defines a state $\f\in\cs(*_J\ga)$, and a representation $\pi$ of $*_J\ga$ on the Hilbert space $\ch$ such that $(\pi,\ch,\Om)$ is the GNS representation of the state $\f$. Conversely, each state $\f\in\cs(*_J\ga)$ defines a stochastic process.

Such one-to-one correspondence sends
exchangeable and stationary processes (provided the set $J$ is $\bz$ for the latter) to symmetric or shift invariant states, respectively.
\end{thm}
\begin{proof}
Take a quadruple $\big(\ga,\ch,\{\iota_j\}_{j\in J},\Om\big)$ and consider the universal property of the free product $C^*$-algebra $*_J\ga$ together with the corresponding $*$-monomorphisms $i_j:\ga\to *_J\ga$, $j\in J$. Then there exists a
$C^*$-homomorphism $\pi:*_J\ga\to\cb(\ch)$, that is a representation making commutative the diagram
\begin{equation*}
\xymatrix{ \ga \ar[r]^{i_j} \ar[d]_{\iota_{j}} &
{*_J\ga}\ar[dl]^\pi \\
\cb(\ch)}\,,\quad j\in J\,.
\end{equation*}
It is easily seen that
$$
\f(X):=\langle\pi(X)\Om,\Om\rangle\,,\quad X\in *_J\ga\,,
$$
defines a state whose GNS representation is precisely $(\pi,\ch,\Om)$, see Theorem 3.3 in \cite{CrFid}. Conversely, for each state $\f\in\cs(*_J\ga)$ with GNS representation $(\pi_\f,\ch_\f,\Om_\f)$, one can define the collection of $*$-homomorphisms $\iota_j:\ga\to \cb(\ch_\f)$ by $\iota_j:=\pi_\f\circ i_j$, $j\in J$.
It is straightforward to check that the quadruple $\big(\ga,\ch_\f,\{\iota_j\}_{j\in J},\Om_\f\big)$ is a stochastic process according to Definition \ref{upun}.

Finally one can see, as in Theorem 3.3 in \cite{CrFid}, that exchangeable or stationary stochastic processes correspond to symmetric or shift invariant states.
\end{proof}
\begin{defin}
\label{abcom}
If $\ga$ is abelian, then the stochastic process is called {\it commutative or classical} if, for the homomorphisms $\iota_j$ in Definition \ref{upun},
$$
\iota_{j_k}(A)\iota_{j_l}(B)=\iota_{j_l}(B)\iota_{j_k}(A)\,,\,\, j_k ,j_l\in J\,, A,B\in\ga\,.
$$
\end{defin}
\vskip0.2cm
\noindent
One immediately recognises Definition \ref{abcom} covers all stochastic processes arising in Classical Probability.

Consider the {\it free abelian product} unital $C^*$-algebra $\ab_J\ga$ of a single, not necessarily abelian $C^*$-algebra $\ga$. It is the universal object among the $C^*$-algebras, for the morphisms with commuting ranges. In other words, if $\{\r_{j}\}_{j\in J}$ is a collection of $*$-homomorphisms such that
$$
\r_{j_1}(A_1)\r_{j_2}(A_2)=\r_{j_2}(A_2)\r_{j_1}(A_1)\,,\quad j_1\neq j_2\,,A_1,A_2\in\ga\,,j_1,j_2\in J\,,
$$
then $\ab_J\ga$ is the universal (unital) $C^*$-algebra making commutative
\begin{equation*}
\xymatrix{ \ga \ar[r]^{r_j} \ar[d]_{\r_{j}} &
{\ab_J\ga}\ar[dl]^{{\rm P}}\\
\gb}\,,\quad j\in J\,,
\end{equation*}
where each $r_j$ is the canonical embedding.
Such a universal object can be described in the following way.

Consider the norm closed two-sided ideal
$$
\gi:=\overline{\big(\text{span}\{a[i_1(A_1),i_2(A_2)]b\mid a,b\in *_J\ga,A_1,A_2\in\ga,i_1,i_2\in J,i_1\neq i_2\}\big)}^{\|\,\,\|}
$$
which is the smallest one containing all commutators in $*_J\ga$ of the form $[i_1(A_1),i_2(A_2)]$, $i_1\neq i_2$. Thus $\ab_J\ga=*_J\ga/\sim$, where the relation $\sim$ is that associated to the closed two-sided ideal $\gi$ above. By
$P:*_J\ga\to\ab_J\ga$ we denote the associated quotient map.

For a fixed unital $C^*$-algebra $\ga$ and a finite subsets $I,I_1,I_2\subset J$, define the $|I|$-times projective $C^*$-tensor product (cf. Section IV.4 of \cite{Tak})
$$
\ga_I:=\underbrace{\ga\otimes_\text{max}\cdots\otimes_\text{max}\ga}_{|I|\,\, \text{times}}
$$
together with the canonical embedding
$$
\ga_{I_1}\sim\ga_{I_1}\otimes\idd_{\ga_{I_2\backslash I_1}}\subset\ga_{I_2}\,,\quad I_1\subset I_2\,.
$$
It is then possible to form the $C^*$-inductive limit (cf. Section L.2 of \cite{WO}) denoted as
$$
\otimes^\text{max}_J\ga:=
\lim_{\longrightarrow}\!{}_{{}_{I\uparrow J}}\ga_I\,.
$$
We point out the fact that all above considerations can be extended to the non unital case described in Section 3 of \cite{CrFid} either by using an approximate unity which always exists in any
$C^*$-algebra, or by adding a unity to $\ga$. The reader is referred to Section 2.2.3 of \cite{BR1} or Section IV.4 of \cite{Tak}.
\vskip.6cm
\begin{rem}
By Proposition IV.4.7 in \cite{Tak}, we have
$$
\ab_J\ga\sim\otimes^\text{max}_J\ga\,.
$$
If in addition $\ga$ is commutative, i.e. $\ga\sim C(I)$ for a compact space $I$, then
$$
\ab_J\ga\sim C\big(\prod_J I\big)
$$
as $C(I)$ is nuclear. We refer the reader to Theorem 36.1 in \cite{Bill} for the explicit construction of the probability measure corresponding to the stochastic process under consideration.
\end{rem}
\vskip0.2cm
\noindent
Recall that an abelian stochastic process uniquely determines a state $\f\in\cs(*_J\ga)$. The incoming result shows how to perform commutative stochastic processes in this picture.
\begin{prop}
\label{propab}
For an abelian $C^{*}$-algebra $\ga$ of samples, a stochastic process $\big(\ga,\ch,\{\iota_j\}_{j\in J},\Om\big)$, and the corresponding state $\f\in\cs(*_J\ga)$ according to Theorem \ref{mmin}, the following are equivalent:
\begin{itemize}
\item[(i)] $\f$ is the pull back on $*_J\ga$ of a state $\om\in\cs(\ab_J\ga)$, i.e. $\f=\om\circ P$;
\item[(ii)] the stochastic process $\big(\ga,\ch,\{\iota_j\}_{j\in J},\Om\big)$ is commutative.
\end{itemize}
\end{prop}
\begin{proof}
We treat the unital case. In absence of unity we can recover the same result arguing as in \cite{CrFid}.

(i) $\Rightarrow$ (ii) If $(\pi_\om,\ch_\om,\Om_\om)$ is the GNS representation of $\om$, the corresponding GNS representation of $\f$ and $*$-homomorphisms $\iota_j$, $j\in J$, are given by $(\pi_\om\circ P,\ch_\om,\Om_\om)$ and $\pi_\om\circ P\circ i_j$, respectively. We compute for each $j_1\neq j_2$,
\begin{align*}
&\iota_{j_1}(A_1)\iota_{j_2}(A_2)=\pi_\om\big(P(i_{j_1}(A_1))\big)\pi_\om\big(P(i_{j_2}(A_2))\big)\\
=&\pi_\om\big(P(i_{j_1}(A_1))P(i_{j_2}(A_2))\big)=\pi_\om\big(P(i_{j_2}(A_2))P(i_{j_1}(A_1))\big)\\
=&\pi_\om\big(P(i_{j_2}(A_2))\big)\pi_\om\big(P(i_{j_1}(A_1))\big)=\iota_{j_2}(A_2)\iota_{j_1}(A_1)\,.
\end{align*}
(ii) $\Rightarrow$ (i) As it is shown in Theorem \ref{mmin}, we get
\begin{equation}
\label{abel}
\f=\langle\pi(\,{\bf\cdot}\,)\Om,\Om\rangle\,,
\end{equation}
where $\pi$ is the unique homomorphism making commutative the diagram
\begin{equation*}
\xymatrix{ \ga \ar[r]^{i_j} \ar[d]_{\iota_{j}} &
{*_J\ga}\ar[dl]^\pi \\
M}\,,\quad j\in J\,,
\end{equation*}
and $M=\bigvee_{j\in J}\iota_j(\ga)$ is the von Neumann algebra acting on $\ch$ generated by all images $\iota_j(\ga)$, $j\in J$. Since $M$ is abelian, for each $j\in J$ and the embeddings
$r_j:\ga\rightarrow \ab_J\ga$, $i_j:\ga\rightarrow *_J\ga$, the universal properties applied to $*_J\ga$ and $\ab_J\ga$ respectively, give
$P\circ i_j=r_j$ and the existence of a unique $\s:\ab_J\ga\rightarrow M$, such that
$\s\circ r_j= \iota_j$, $j\in J$. Then $\s\circ P\circ i_j=\iota_j$, and consequently, $\pi=\s\circ P$. If
$\om:=\langle\s(\,{\bf\cdot}\,)\Om,\Om\rangle\in\cs(\ab_J\ga)$,
\eqref{abel} gives
$$
\f=\langle\pi(\,{\bf\cdot}\,)\Om,\Om\rangle=\langle\s\circ P(\,{\bf\cdot}\,)\Om,\Om\rangle
=\om\circ P\,.
$$
\end{proof}
\noindent
We have just shown that the general quantum scenario described in the first part of the section includes, as a particular case, the classical scheme. The latter is indeed achieved by the commutative diagram
\begin{equation}
\label{unni1}
\xymatrix{ \ga \ar[r]^{i_j} \ar[d]_{\b_{j}} &
{*_J\ga}\ar[dl]^\Phi \\
\gb}\,,\quad j\in J\,,
\end{equation}
where $\gb=\ab_J\ga$, $\b_{j}=r_j$, $j\in J$, and $\Phi=P$ are the canonical embeddings of $\ga$ in $\ab_J\ga$ and the canonical projection of $*_J\ga$ onto its abelianised
$\ab_J\ga$, respectively. Thus any $\f\in \cs(*_J\ga)$ realising a classical stochastic process (cf. Theorem \ref{mmin}) is obtained by a state $\om$ on $\ab_J\ga$ through a pull back relation.

The basic idea yielding the above result is taking a suitable quotient of the free $C^*$-algebra. Hence it appears clear that in the general case there are several ways to consider processes, each of them arising from factoring out $*_J\ga$ by two-sided ideals generated by suitable commutators, and the commutative diagram \eqref{unni1} can be seen as the most general situation describing quantum stochastic processes.
To get a flavour we mention the so called $q$-deformed relations for $q\in [-1,1]$, with pivotal examples given by $q=\pm 1$ corresponding to the Bose/Fermi cases, and $q=0$ corresponding to the Boltzmann case describing the group reduced $C^{*}$-algebra of the free group on infinitely many generators \cite{BS}, see Section \ref{41} for further details. Other noteworthy cases are the Monotone \cite{CFL} and the Boolean cases \cite{CrFid, F}. As possible future investigations we also mention the cases arising from the more general setting of interacting Fock spaces, see e.g. \cite{AB}.

Concerning the Bose case (cf. Section \ref{42}), consider the infinite tensor product $C^*$-algebra $\otimes^\text{min}_J\ga$ as in \cite{St2} for the non necessarily abelian algebra of samples $\ga$, together with the canonical projection $\F:*_J\ga\to\otimes^\text{min}_J\ga$ recovered by universality by the embeddings $t_j:\ga\to\otimes^\text{min}_J\ga$, $j\in J$. As the $t_j$ have commuting ranges, $\F$ factors through
the canonical projection $\Psi:\ab_J\ga\to\otimes^\text{min}_J\ga$, i.e. $\F=\Psi\circ P$. Accordingly, the stochastic process determined by a state $\om\in\cs(\otimes^\text{min}_J\ga)$ such that $\f=\om\circ\F$, factors through $\ab_J\ga$, as $\f=\om\circ\Psi\circ P$. We have then the following
\begin{rem}
Each stochastic process on $\otimes^\text{min}_J\ga$
comes from a stochastic process on the free abelianised product $\ab_J\ga\sim\otimes^\text{max}_J\ga$ uniquely determined by the state $\om\circ\Psi\in\cs(\ab_J\ga)$. The same construction holds true for stochastic processes on any other infinite tensor product $\otimes^\g_J\ga$ based on the $C^*$-cross norm $\|\,\,\|_\g$, see Section IV.4 of \cite{Tak}.
\end{rem}
We end the section recalling the definition, useful in the sequel, of the {\it tail algebra} $\gz^\perp_\f$ for the stochastic process $(\ga, \ch, (\iota_j)_{j\in J}, \Om)$, with corresponding state $\f\in\cs(*_J\ga)$
$$
\gz^\perp_\f:=\bigwedge_{\begin{subarray}{l}I\subset J,\,
I \text{finite} \end{subarray}}\left(\bigcup_{\begin{subarray}{l}K\bigcap I=\emptyset,
\\\,\,\,K \text{finite} \end{subarray}}\left( \bigvee_{k\in K}\iota_k(\ga)\right)\right)''
$$
In Statistical Mechanics it is known as the {\it algebra at infinity}, see e.g. \cite{BR1}.

\section{ergodic properties of stochastic processes}
\label{sec3}

The present section is devoted to the investigation of natural ergodic properties of stochastic processes. Here is also performed a direct link between algebraic relations and ergodic conditions.

Let $(\ga, \a)$ be a $C^*$-dynamical system with $\cs_{G}(\ga)=\{\om\}$. It is said to be {\it uniquely ergodic}. When $G=\bz$, one can see that unique ergodicity is equivalent to
\begin{equation}
\label{eav}
\lim_{n\rightarrow+\infty} \frac{1}{n}\sum_{k=0}^{n-1}f(\alpha^k(a))=f(\idd)\omega(a)\,, \quad a\in \ga\,,f\in\ga^*\,,
\end{equation}
or again to
$$
\lim_{n\rightarrow+\infty} \frac{1}{n}\sum_{k=0}^{n-1} \alpha^k(a)=\omega(a)\idd, \,\,\,\,\,\, a\in \ga\,\,,
$$
pointwise in norm. Some natural generalisations of such a strong ergodic property can be achieved by replacing the ergodic average \eqref{eav} with
$$
\lim_{n\rightarrow+\infty} \frac{1}{n}\sum_{k=0}^{n-1}|f(\alpha^k(a))-f(\idd)\omega(a)|=0\,, \quad a\in \ga\,,f\in\ga^*\,,
$$
or simply
$$
\lim_{n\rightarrow+\infty} f(\alpha^n(a))=f(\idd)\omega(a)\,, \quad a\in \ga\,,f\in\ga^*\,,
$$
for some state $\om\in\cs(\ga)$ which is necessarily invariant. In this case, $(\ga,\a)$ is called {\it uniquely weak mixing} or {\it uniquely mixing}, respectively. For all these cases,
$\ga^\bz=\bc\idd$, and the (unique) invariant conditional expectation onto the fixed point subalgebra is precisely $E(a)=\om(a)\idd$.

Another natural generalisation is to look at the fixed point subalgebra whenever it is nontrivial, and at the unique invariant conditional expectation onto such a subalgebra
$E^\bz:\ga\to\ga^\bz$, provided the last exists. The unique ergodicity, weak mixing, and mixing w.r.t. the fixed point subalgebra (denoted also as $E^\bz$-ergodicity, $E^\bz$-weak mixing and $E^\bz$-mixing, $E^\bz$ being the invariant conditional expectation onto $\ga^\bz$ which necessarily exists) are given by definition, for $a\in\ga$ and $f\in\ga^*$, by
\begin{align*}
&\lim_{n\rightarrow+\infty} \frac{1}{n}\sum_{k=0}^{n-1}f(\alpha^k(a))=f(E^\bz(a))\,,\\
&\lim_{n\rightarrow+\infty} \frac{1}{n}\sum_{k=0}^{n-1}|f(\alpha^k(a))-f(E^\bz(a))|=0\,,\\
&\lim_{n\rightarrow+\infty} f(\alpha^n(a))=f(E^\bz(a))\,.
\end{align*}
Fix a dynamical system $(\ga,\a)$ based on the group $G$. For each state $\f\in\cs(\ga)$, we put $L^\infty(\ga,\f):=\pi_\f(\ga)''$, and $L^2(\ga,\f):=\pi_\f(\ga):=\ch_\f$. One can define a bounded linear map $T:L^\infty(\ga,\f)\to L^2(\ga,\f)$ given by
$$
TX:=X\Om_\f\,,\quad X\in L^\infty(\ga,\f)\,.
$$
Suppose $\f\in\cs_G(\ga)$. One can always denote by $\a$ the actions of $G$ on both $L^\infty(\ga,\f)$ and $L^2(\ga,\f)$ as $\a_g(X):=\ad_{U_\f(g)}(X)$ or $\a_g(\xi):=U_\f(g)\xi$. Such actions are compatible with the map $T$:
$$
T\ad\!{}_{U_{\f}(g)}(X)=U_\f(g)TX\,,\quad g\in G\,, X\in L^\infty(\ga,\f)\,.
$$
The set of the invariant elements are denoted respectively as
\begin{align*}
L^\infty(\ga,\f)^G:=&\pi_\f(\ga)''\bigwedge \{U_\f(G)\}'\,,\\
L^2(\ga,\f)^G:=\{\xi\in&\ch_\f\mid U_\f(g)\xi=\xi\,, g\in G\}\,.
\end{align*}
\noindent
\begin{defin}
A state $\f\in\cs_{G}(\ga)$ is said {\it weakly clustering} if
\begin{equation}
\label{eq1}
\lim_{\L\uparrow G}
\frac1{|\L|}\sum_{g\in\L}\f(A\a_g(B))=\f(A)\f(B)\,,\,\,\,A,B\in\ga
\end{equation}
\noindent
and it satisfies the {\it property of the convergence to the equilibrium} if
\begin{equation}
\label{eq11}
\lim_{\L\uparrow G}
\frac1{|\L|}\sum_{g\in\L}\f(A\a_g(B)C)=\f(AC)\f(B)\,,\,\,\,A,B,C\in\ga
\end{equation}
along the net $\{\L\subset G\mid |\L|<\infty\}$
\end{defin}
\vskip0.2cm
\noindent
We notice that, when $G=\bz$, i.e. one deals with the shift, one usually reduces the matter to
$\L_n:=[0,n-1]$. Condition \eqref{eq1} is obviously weaker than \eqref{eq11}. We can see that those are equivalent provided the support of $\f$ belongs to the centre of the bidual, i.e. $s(\f)\in Z(\ga^{**})$. They are equivalent also for (graded) asymptotically abelian states, that is for systems possibly including Fermions (cf. \cite{BF2, F16}), and in classical case. Furthermore a weakly clustering state $\f\in\cs_G(\ga)$ is automatically ergodic (i.e. extremal invariant), that is $\f\in\ce(\cs_G(\ga))$. The converse holds true in the particular case of $G$-abelian states, see
Section 4.3 of \cite{BR1}.

It is our aim, in the forthcoming lines, to prove that condition \eqref{eq11} is the right one ensuring the convergence to the equilibrium
in quantum case. Thus consider a physical system in a state $\om$ invariant for the (discrete) dynamics.
Then we have a dynamical system $(\ga, \a,\om)$ based on a single automorphism $\a$ on $\ga$. Such a localised perturbation usually produces a state $\f$ which is normal w.r.t. the reference state $\om$. Namely,
$\f\in\cf_{\pi_\om}(\ga)$ where, for each representation $(\pi,\ch_\pi)$ of $\ga$,
$$
\cf_{\pi}(\ga)=\{\f\in\cs(\ga)\mid\f=\tr(\pi(\,{\bf\cdot}\,)T)\,,T\in\ct(\ch_\pi)_{+,1}\}
$$
is the folium generated by $\pi$,
$\tr$ and $\ct(\ch_\pi)_{+,1}$ being the canonical trace on $\ch_\pi$ and the positive normalised trace class operators on it, respectively. The convergence to the equilibrium simply means
\begin{equation}
\label{equilibrium}
\lim_n\frac1{n}\sum_{k=0}^{n-1}\f(\a^n(A))=\om(A)\,,\quad A\in\ga\,,\f\in\cf_{\pi_\om}(\ga)\,,
\end{equation}
that is $\cf_{\pi_\om}(\ga)$ is contained in the {\it basin of attraction} of $\om$ in the whole $\cs(\ga)$.
By a standard approximation argument, we can reduce the matter to the generators of $\cf_{\pi_\om}(\ga)$ of the form
$$
\f_B(A):=\frac{\om(B^*AB)}{\om(B^*B)}\,,\quad A\in\ga\,,
$$
provided $\om(B^*B)\neq0$. It appears then evident the condition
\eqref{equilibrium} simply means
$$
\lim_n\frac1{n}\sum_{k=0}^{n-1}\om(B^*\a^n(A)B)=\om(B^*B)\om(A)\,,\quad A,B\in\ga
$$
which turns out to be equivalent to \eqref{eq11} by polarisation.

We report a suitable version of a pivotal result in noncommutative Ergodic Theory, well known to the experts. To avoid technicalities, from now on we reduce the matter to $\bp_J$ or $\bz$ if it is not otherwise specified.
Recall that for a $C^*$-algebra $\ga$, for $\f\in\cs(\ga)$ with support $s(\f)\in\ga^{**}$ and the corresponding GNS representation $(\pi_\f,\ch_\f,\Omega_\f)$, one has $s(\f)\in Z(\ga^{**})$ if and only if
$\Omega_\f$ is cyclic for $\pi_\f(\ga)'$.
\begin{thm}
\label{eeerg}
Let $(\ga,\a)$ be a dynamical system where $G$ is $\bp_J$ or $\bz$. With the net generated by all finite subgroups $\bp_I$, $\{I\subset J\mid |I|<\infty\}$ for $\bp_J$, or the sequence $[0,n-1]$,
$n\in\bn$ for $\mathbb{Z}$ respectively, and $\f\in\cs_{G}(\ga)$, consider the following assertions:
\begin{itemize}
\item[(i)] $\f$ is weakly clustering,
\item[(ii)] $L^2(\ga,\f)^G=\bc\Om_\f$,
\item[(iii)] $\f$ satisfies the property of the convergence to the equilibrium,
\item[(iv)] $L^\infty(\ga,\f)^G=\bc I$.
\end{itemize}
Then we have ${\rm(iv)}\Longrightarrow{\rm(iii)}\Longrightarrow{\rm(ii)}\iff{\rm(i)}$, and all conditions are equivalent if $s(\f)\in Z(\ga^{**})$.
\end{thm}
\begin{proof}
The equivalence ${\rm(i)}\iff{\rm(ii)}$ is nothing but the Mean Ergodic Theorem of J. von Neumann, see e.g. Section 4.3 of \cite{BR1}, and Proposition 3.1 of \cite{CrF}. The equivalence
${\rm(ii)}\iff{\rm(iv)}$ is well known (cf. Section 4.3 of \cite{BR1}), provided $s(\f)\in Z(\ga^{**})$.  We now report ${\rm(iv)}\Longrightarrow{\rm(iii)}$.

By reasoning as in Theorem 4.3.20 of \cite{BR1}, first we note that ${\rm(iv)}$ implies that $s(\f)\in Z(\ga^{**})$. In addition, ${\rm(iv)}\Longrightarrow{\rm(ii)}$, which is equivalent to (i). The latter turns out to be equivalent to (iii) as $s(\f)\in Z(\ga^{**})$.
\end{proof}
\noindent
In the following definition we present the algebraic properties of a stochastic process on $\ga$, or equivalently of the corresponding state on $*_J\ga$, which will be linked with the above ergodic conditions. To achieve also the shift on the chain, we specialise the matter to $*_\bz\ga$.
\begin{defin}
\label{bspsc}
The state $\f\in\cs(*_\bz\ga)$ is said to satisfy the \emph{product state condition} if
$$
\f(A_1A_2)=\f(A_1)\f(A_2)\,,
$$
whenever $A_k\in{\rm alg}\{i_{j_k}(\ga)\mid j_k\in I_k\}$, $I_k\subset J$, $k=1,2$, and
$\quad I_1\cap I_2=\emptyset$.

The state $\f$ satisfies the {\it block singleton condition} \cite{ABCL} if
$$
\f(A_1A_2A_3)=\f(A_1A_3)\f(A_2)\,,
$$
whenever $A_k\in{\rm alg}\{i_{j_k}(\ga)\mid j_k\in I_k\}$, $I_k\subset J$, $k=1,2,3$, and
$(I_1\cup I_3)\cap I_2=\emptyset$.
\end{defin}
\vskip0.2cm
\noindent
The next result glues the above conditions with those described in Theorem \ref{eeerg}. We refer the reader to \cite{CrFid} for further details.
\begin{thm}
\label{pscm}
For $G$ as in Theorem \ref{eeerg} and $\f\in\cs_{G}(*_\bz\ga)$, the following assertions hold true.
\begin{itemize}
\item[(i)] $\f$ satisfies the product state condition if and only if it is weakly clustering,
\item[(ii)] $\f$ is a block singleton state if and only if it satisfies the property of the convergence to the equilibrium.
\end{itemize}
\end{thm}
\begin{proof}
We report the case of the permutations (cf. \cite{CrFid}) corresponding to (ii), and leave to the reader the easier situation (i). The analogous case of the shift is in \cite{ABCL}.

Suppose that $\f\in\cs_{\bp_\bz}(*_\bz\ga)$ satisfies the property of the convergence to the equilibrium. Fix words $u,v,w\in*_\bz\ga$ such that their respective supports satisfy
$I_v\cap(I_u\cup I_w)=\emptyset$. Take $I\supset I_u,I_v, I_w$ an arbitrary large but finite part of $J$, and consider the set $B\subset \bp_I$ of the permutations leaving $I_u\cup I_w$ pointwise fixed. Since $\f$ is symmetric, we get
\begin{align*}
&\f(uvw)=\frac1{|B|}\sum_{g\in B}\f(\a_g(uvw))=\frac1{|B|}\sum_{g\in B}\f(u\a_g(v)w)\\
=&\frac{|\bp_I|}{|B|}\bigg(\frac1{|\bp_I|}\sum_{g\in\bp_I}\f(u\a_g(v)w)\bigg)
-\frac1{|B|}\sum_{g\in\bp_I\backslash B}\f(u\a_g(v)w)\,.
\end{align*}
By Lemma 3.3 in \cite{CrF}, $\frac{|B|}{|\bp_I|}\rightarrow 1$ and $\frac{|B^c|}{|B|}\rightarrow 0$, as
$I\uparrow J$. By taking the limit from both parts, the l.h.s. does not depend on $J$, whereas the r.h.s. converges to $\f(uw)\f(v)$ by the property of the convergence to the equilibrium. Thus $\f$ satisfies the block singleton condition.

Conversely, suppose now $\f\in\cs_{\bp_\bz}(*_\bz\ga)$ is a block singleton state. By density, we can reduce the matter to elementary words. Fix
$u,v,w\in*_\bz\ga$ with supports
$I_u,I_v, I_w$ respectively.
If $I$ is a finite part of $J$, define
$$
A:=\{g\in\bp_I\mid (I_u\cup I_w)\cap I_{\a_g(v)}=\emptyset\}\,.
$$
By applying the block singleton condition, we get
\begin{align*}
&\bigg|\frac1{|\bp_I|}\sum_{g\in\bp_I}\f(u\a_g(v)w)-\f(uw)\f(v)\bigg|\\
\leq&\bigg|\frac1{|\bp_I|}\sum_{g\in A}\f(u\a_g(v)w)- \f(uw)\f(v)\bigg|
+\bigg|\frac1{|\bp_I|}\sum_{g\in\bp_I\backslash A}\f(u\a_g(v)w)\bigg|\\
\leq&\bigg|\bigg(\frac{|A|}{|\bp_I|}-1\bigg)\f(uw)\f(v)\bigg|+\|u\|\|v\|\|w\|\frac{|A^c|}{|\bp_I|}\,,
\end{align*}
where $A^c:= \bp_I\backslash A$. Taking the limit $I\uparrow J$, again by Lemma 3.3 of \cite{CrF} one has that $\frac{|A|}{|\bp_I|}\to 1$ and
$\frac{|A^c|}{|\bp_I|}\rightarrow 0$. Thus $\f$ satisfies condition \eqref{eq11}.
\end{proof}
\noindent
As an application of the above results, one achieves some ergodic properties for the so-called Haagerup states \cite{Ha} naturally arising in Free Probability. For $\l\in(0,+\infty)$, those are defined on the group $C^*$-algebra $C^*(\bbf_\infty)$ of the free group $\bbf_\infty$ on infinitely many generators, as
$\f_\l(w):=e^{-\l|w|}$, where $w\in C^*(\bbf_\infty)$ is a reduced word and $|w|$ is its length. The case $\l=+\infty$ corresponds to the tracial state, the unique one passing to the quotient given by the reduced group algebra $C_{\rm red}^*(\bbf_\infty)$ and considered below.
The Haagerup states are automatically symmetric by construction, and satisfy the product state condition:
$\f_\l(vw)=\f_\l(v)\f_\l(w)$, $I_v\cap I_w=\emptyset$. But they do not fulfil the block singleton condition if $\l\in(0,+\infty)$. In fact, for elementary generators $g_i,g_j$ with $i\neq j$,
$$
\f_\l(g_ig_jg_i^{-1})=e^{-3\l}\neq e^{-\l}
=\f_\l(g_j)\f_\l(\idd)=\f_\l(g_j)\f_\l(g_ig_i^{-1})\,.
$$
\begin{cor}
For the Haagerup states $\f_\l\in\cs(C^*(\bbf_\infty))$ one has $s(\f_\l)\not\in Z(C^*(\bbf_\infty)^{**})$, $\l\in(0,+\infty)$.
\end{cor}
\begin{proof}
As $\f_\l$ satisfies the product state condition, by Theorem \ref{pscm} it is weakly clustering. Suppose that $s(\f_\l)\in Z(C^*(\bbf_\infty)^{**})$. By Theorem \ref{eeerg}, it satisfies the property of the convergence to the equilibrium. Then it is a block singleton state again by Theorem \ref{pscm}, which is a contradiction.
\end{proof}

\section{examples}
\label{sec4}
In this section we deal with stochastic processes directly built on concrete $C^*$-algebras.


\subsection{$q$-deformed commutation relations}
\label{41}
We briefly recall the $q$-deformed commutation relations, $q\in[-1,1]$:
\begin{equation}
\label{cre}
a_q(i)a_q^\dagger(j)-qa_q^\dagger(j)a_q(i)=\d_{ij}\idd\,,\,\, i,j\in\bz\,.
\end{equation}
The above commutation rule can be represented concretely as creators and annihilators on the $q$-deformed Fock spaces, see e.g. \cite{BS}. The remarkable cases of Bose (CCR), Fermi (CAR) and Boltzmann (Free) relations are realised for $q=\pm 1$ and $q=0$, respectively. We first treat the case $q\in(-1,1)$.

Let $\gar_q$ and $\gg_q$ the concrete unital $C^*$-algebras acting to the $q$-Fock space generated by the annihilators $\{a_q(i)\mid i\in\bz\}$, and by their selfadjoint part
$\{a_q(i)+a^\dagger_q(i)\mid i\in\bz\}$, respectively, for $a_q(i):=a_q(e_i)$ and $e_i$ the elements of the canonical basis of $\ell^2(\mathbb{Z})$. The group of the permutations
$\bp_\bz$, and the group $\bz$ generated by the powers of the shift naturally act on both $\gar_q$ and $\gg_q$ as Bogoliubov automorphisms implemented by the unitaries $Ue_i:=e_{i+1}$ and $U_ge_i:=e_{g(i)}$ on $\ell^2(\bz)$ (see \cite{CFL}, Proposition 3.1). We denote by $\ga_q$ one of these concrete $C^*$-algebras, and by
$G$ and $\a$ those groups and their actions, respectively. The vacuum expectation state is given by $\om_q:=\langle\,\cdot\,\Om_q, \Om_q\rangle$, $\Om_q$ being the vacuum vector in the
$q$-Fock space.

We report the following strong ergodic result, and refer the reader to \cite{CFL} for its proof.
\begin{thm}
Fix $q\in(-1,1)$ and consider a countable set $\{g_k\}_{k\in\bn}\subset G$. Then
$$
\lim_{n}\bigg\|\frac1{n}\sum_{k=1}^n\a_{g_k}(A)-\om_q(A)I\bigg\|=0\,,\quad A\in\ga_q\,.
$$
In addition,
$$
\cs_{\bp_\bz}(\ga_q)=\cs_{\bz}(\ga_q)=\{\om_q\}\,.
$$
The $C^*$-dynamical system $(\ga_q,\a)$ based on the action of $\bz$ of powers of the shift
is uniquely mixing with $\om_q$ as the unique invariant state.
\end{thm}
\noindent
As the case $q=0$ corresponds to the reduced group algebra of the free group $\gg_0\sim C_{\rm red}^*(\bbf_\infty)$ (cf. \cite{VDN}),
with $\om_0$ corresponding to the canonical trace, we then achieve $\om_0$ as the unique state on $C_{\rm red}^*(\bbf_\infty)$ invariant for both the permutations moving only finitely many generators, and the shift. In addition, it is the unique invariant state on $C^*(\bbf_\infty)$ coming from the natural quotient
$$
C_{\rm red}^*(\bbf_\infty)=C^*(\bbf_\infty)/{\rm ker\l}
$$
$\l$ being the (left) regular representation of $\bbf_\infty$.

\subsection{Bose case}
\label{42}

As the creators satisfying \eqref{cre} cannot be bounded if $q=1$,
we manage the Boson case by using the Weyl algebra (formally by exponentiating the field operators, see e.g. \cite{BR1}).  In this situation, $\gar_1$ must be replaced by the Weyl algebra $W(C_{00}(\bz))$, where $C_{00}(\bz)$ is the pre-Hilbert space of all finitely supported complex sequences on $\bz$. The algebra generated by the selfadjoint parts of annihilators leads to abelian processes and is treated in the standard literature of Probability. It is well known that
$$
W(C_{00}(\bz))\sim\bigotimes_{_\bz}{}_{{}_{\rm min}}W(\bc)\,,
$$
the infinite tensor product of infinitely many copies of $W(\bc)$.
By Stormer's results \cite{St2}, one can obtain an ergodic decomposition of symmetric states as in de Finetti Theorem:
\begin{itemize}
\item[(i)] $\cs_{\bp_\bz}(W(C_{00}(\bz)))$ is a mixture
(i.e. direct integral) of states which are an infinite product state of a single one on $W(\bc)$, the latter providing the
ergodic ones $\ce(\cs_{\bp_\bz}(W(C_{00}(\bz))))$.
\end{itemize}
By using the results in \cite{AF1}, it is not hard to show that
\begin{itemize}
\item[(ii)] $\cs_{\bp_\bz}(W(C_{00}(\bz)))\subsetneq\cs_{\bz}(W(C_{00}(\bz)))$.
\end{itemize}
One can also prove the following version of de Finetti Theorem:
\begin{itemize}
\item[(iii)] a process on the Weyl algebra is exchangeable if and only if it is conditionally independent and identically distributed w.r.t. the tail algebra.
\end{itemize}
Finally, as in Theorem 5.3 of \cite{CrFid} one can establish a quantum analogue of the Hewitt and Savage Lemma:
\begin{itemize}
\item[(iv)] for states $\f\in\cs_{\bp_\bz}(W(C_{00}(\bz)))$, the {\it tail algebra} $\gz^\perp_\f$ coincides with the symmetric part of the centre $\gz^{\bp_\bz}_\f$
\end{itemize}

\subsection{Fermi case} For the unital $C^*$-algebra $\gar_{-1}$ generated by the annihilators (i.e. the CAR algebra), the same results listed above for the the Bose case hold true. The reader is referred to
\cite{CrF, CrFid}, and the examples relative to Fermi Markov states in Section 6 in \cite{F17} for the point (ii). As any symmetric state on the CAR algebra is automatically shift invariant, it is even (cf. \cite{BR1}). Then the analogue of (iv) above (the Hewitt and Savage Lemma for the Bose case) assumes the following form in the CAR case (cf. Theorem 5.3 of \cite{CrFid}):
\begin{itemize}
\item[(iv')] for states $\f\in\cs_{\bp_\bz}(\gar_{-1})$, the tail algebra $\gz^\perp_\f$ coincides with the even portion of the symmetric part of the centre.
\end{itemize}

\subsection{Boolean case}  Let $\ch$ be a complex Hilbert space. The Boolean Fock space over $\ch$ is given by $\G(\ch):=\mathbb{C}\oplus \ch$ and $(1,0)$ is the vacuum vector. On $\Gamma(\ch)$ we define the creation and annihilation operators, respectively given for $f,g\in \ch$ and $\a\in\bc$ by
$$
a^\dagger(f)(\alpha\oplus g):=0\oplus \alpha f,\,\,\,\, a(f)(\alpha\oplus g):=\langle g,f\rangle_\ch \oplus 0\,.
$$
For $\ch=\ell^2(\bz)$, it is seen that the concrete unital $C^*$-algebra $\gb$ (called the Boolean algebra) generated by the annihilators coincides with that generated by their selfadjoint parts, see e.g.
\cite{CrFid}. In addition,
$$
\gb=\ck(\ell^2(\{\#\}\cup\bz))+\bc I\,.
$$
Here, $a_i:=a(e_i)=\varepsilon_{\#,i}$ is the standard matrix unit. Here $\#$ corresponds to the subspace in $\G(\ell^2(\bz))$ generated by the vacuum $e_\#$,
and "$\ck$" stands for compact operators. If $\om_\#$ denotes the vacuum state and $\om_\infty$ the state at infinity:
$$
\om_\infty(A+cI):=c,\,\,\,\, A\in\ck(\ell^2(\{\#\}\cup\bz)),\,\,\,c\in\bc\,,
$$
we get the following structure for symmetric and stationary states (cf. \cite{CrFid, CFL}):
\begin{thm} For the shift-invariant and symmetric states, we get
$$
\cs_{\bp_\bz}(\gb)=\cs_{\bz}(\gb)=\{(1-\g)\om_\#+\g\om_\infty\}\,.
$$
\end{thm}
\noindent
The well established structure of $\gb$ allows to completely determine the fixed point algebras for the action of the shift and the permutations:
$$
\gb^{\bp_\bz}=\gb^{\bz}= \mathbb{C}P_\# \oplus \mathbb{C}P_\#^\bot\,,
$$
$P_\#$ being the orthogonal projection onto $\mathbb{C}e_\#$.

Consider the dynamical system $(\gb,\a)$ based on the Boolean algebra and the shift, together with the unique invariant conditional expectation
$E$ onto $\gb^{\bp_\bz}$ given by
\begin{equation}
\label{cescarz}
E(A+bI):=\langle Ae_\#,e_\#\rangle P_\#+bI\,,\quad A\in\ck(\ell^2(\{\#\}\cup\bz))\,,b\in\bc\,.
\end{equation}
\begin{prop} (\cite{CFL}, Proposition 7.2)
\label{bshi}
The $C^*$-dynamical system $(\gb, \a)$ is $E^\bz$-mixing with $E=E^\bz$ the unique invariant conditional expectation onto the fixed point subalgebra given in \eqref{cescarz}.
\end{prop}
\noindent
Notice that the conditional expectation in  \eqref{cescarz} is also the unique invariant one for the natural action of $\bp_\bz$. Denoting again by $\a$ such an action, one can show that
$$
\lim_{J\uparrow\bz}\frac1{|J|!}\sum_{g\in\bp_J}\a_g(A)=E(A)\,,\quad A\in\gb\,,
$$
where $\{J\mid J\subset\bz\}$ is the direct net of all finite subsets of $\bz$.

Moreover, we report the following assertions proved in \cite{F}:
\begin{itemize}
\item[(i)] a Boolean process is exchangeable if and only if it is conditionally independent and identically distributed w.r.t. the tail algebra (as in the classical case);
\item[(ii)]  for $\om\in\cs_{\bp_\bz}(\gb)$, if $\gb_\om:=\pi_\om(\gb)''$ and $\gz^\perp_{\om}$, $\gb_\om^{\bp_\bz}$, $\gb_\om^{\bz}$ denote the tail algebra, the symmetric and stationary ones respectively,
we get
$$
\gz^\perp_{\om}=\gb_\om^{\bp_\bz}\subsetneq\gb_\om^{\bz}
$$
\end{itemize}
As a consequence, the equality above marks the transposition of the Hewitt-Savage Lemma to the Boolean situation, whereas the last inclusion entails the Olshen Theorem \cite {O} does not hold for Boolean stochastic processes.

\subsection{Monotone case}

We outline the structure of the stationary Monotone processes corresponding to states on the concrete unital Monotone $C^*$-algebra, and in addition on the subalgebra generated by the selfadjoint parts of annihilators.

As in \cite{Boz2}, for $k\geq 1$, denote $I_k:=\{(i_1,i_2,\ldots,i_k) \mid i_1< i_2 < \cdots <i_k, i_j\in \mathbb{Z}\}$, and for $k=0$, we take $I_0:=\{\emptyset\}$, $\emptyset$ being the empty sequence. The Hilbert space $\ch_k:=\ell^2(I_k)$ is precisely the $k$-particles space for the monotone quantisation. In particular, the $0$-particle space $\ch_0=\ell^2(\emptyset)$ is identified with the complex scalar field $\mathbb{C}$. The monotone Fock space is $\cf_m=\bigoplus_{k=0}^{\infty} \ch_k$.

For an increasing sequence $\a=(i_1,i_2,\ldots,i_k)$ of integers, we denote by $e_\a$ the generic element of canonical basis of $\cf_m$. There is a natural order structure on such sequences. Indeed, if
$\a=(i_1,i_2,\ldots,i_k)$, $\b=(j_1,j_2,\ldots,j_l)$, we say $\a < \b$ if $i_k < j_1$.
The monotone creation and annihilation operators are respectively given, for any $i\in \mathbb{Z}$, by
\begin{equation*}
a^\dagger_ie_{(i_1,i_2,\ldots,i_k)}:=\left\{
\begin{array}{ll}
e_{(i,i_1,i_2,\ldots,i_k)} & \text{if}\, i< i_1\,, \\
0 & \text{otherwise}\,, \\
\end{array}
\right.
\end{equation*}
\begin{equation*}
a_ie_{(i_1,i_2,\ldots,i_k)}:=\left\{
\begin{array}{ll}
e_{(i_2,\ldots,i_k)} & \text{if}\, k\geq 1\,\,\,\,\,\, \text{and}\,\,\,\,\,\, i=i_1\,,\\
0 & \text{otherwise} \\
\end{array}
\right.
\end{equation*}
where $a_i:=a(e_i)$. Moreover one can prove $\|a^\dagger_i\|=\|a_i\|=1$ and check that $a^\dagger_i$ and $a_i$ are mutually adjoint. The following conditions
\begin{equation*}
\begin{array}{ll}
  a^\dagger_ia^\dagger_j=a_ja_i=0 & \text{if}\,\, i\geq j\,, \\
  a_ia^\dagger_j=0 & \text{if}\,\, i\neq j
\end{array}
\end{equation*}
hold true and, in addition, the following commutation relation
\begin{equation}
\label{reba}
a_ia^\dagger_i=I-\sum_{k\leq i}a^\dagger_k a_k
\end{equation}
is satisfied, with the sum meant in the strong operator topology (cf. Proposition 3.2 in \cite{CFL}).
The $C^*$-algebra $\gam$ acting on $\cf_m$ is the unital
$C^*$-algebra generated by the annihilators $\{a_i\mid i\in\mathbb{Z}\}$. It was proven in \cite{CFL} that the selfadjoint part of
annihilators $\{a_i+a^+_i\mid i\in\mathbb{Z}\}$ also generate the same unital $C^*$-algebra as the annihilators (the circumstance is the same as the Booleans). Thus we can reduce the matter of our investigation to $\gam$.
Because of the order structure, the group $\bp_\bz$ of the permutations does not naturally act on $\gam$. So we are mainly focused on the action of the shift.

The results we enumerate below heavily rely on writing the algebraic part of the algebra (i.e. that algebraically generated by the annihilators) by reduced words in quasi-Wick order, see Section 5 of \cite{CFL}.
If $\gam_0$ is the concrete unital $*$-algebra generated by the monotone annihilators,
a word $X$ in $\gam_0$ is said to have a $\lambda$-\textbf{form} if there are $m,n\in\left\{  0,1,2,\ldots
\right\}$ and $i_1<i_2<\cdots < i_m, j_1>j_2> \cdots > j_n$ such
that
$$
X=a_{i_1}^{\dagger}\cdots a_{i_m}^{\dagger} a_{j_1}\cdots a_{j_n}\,,
$$
with $X=I$, the empty word corresponding to $m=n=0$.
Its length is $l(X)=m+n$.
In addition, $X$ is said to have a $\pi$-\textbf{form} if there are $m,n\in\left\{0,1,2,\ldots
\right\}$, $k\in\mathbb{Z}$, $i_1<i_2<\cdots < i_m, j_1>j_2> \cdots > j_n$ such that $i_m<k>j_1$ and
$$
X=a_{i_1}^{\dagger}\cdots a_{i_m}^{\dagger} a_{k}a_{k}^{\dagger} a_{j_1}\cdots a_{j_n}\,.
$$
Its length is $l(X)=m+2+n$.
As it is seen in \cite{CFL}, first the words in $\l$-form and in $\pi$-form are reduced, and in addition, each element in $\gam_0$ can be expressed as a finite linear combination of $\l$-forms and/or
$\pi$-forms. One could further imagine the set of words in $\l$ and $\pi$-form are linearly independent. Indeed this is not true.
As an example, the reduced $\pi$-form $a^{\dagger}_ia_ja^{\dagger}_ja_l$, $i<j>l$, can be written as sums of $\l$-forms:
$$
a^{\dagger}_ia_ja^{\dagger}_ja_l= a^{\dagger}_ia_l - \sum_{k=(i\vee l) +1}^j a^{\dagger}_ia^{\dagger}_ka_ka_l\,,
$$
where $i\vee l:= \max\{i,l\}$, as one can straightforwardly see by using \eqref{reba}.

The $\l$ and $\pi$-forms structure of the algebra yields a "splitting" representation of $\gam$, which turns out to describe the convex set of stationary states. More in detail, if
$\cam_0:=\spn\big\{X\in\gam_0\mid l(X)>0\big\}$ and $\cam:=\overline{\cam_0}$, where the closure is meant in the norm topology, one finds (cf. \cite{CFL}, Corollary 5.10)
$$
\gam=\cam+\bc I\,.
$$
Furthermore, if $\om$ denotes the vacuum expectation, and $\om_\infty\in \cs(\gam)$ is
$$
\om_\infty(X +cI):=c\,,\quad X\in \cam\,,c\in \mathbb{C}\,,
$$
the monotone stationary states are exactly those lying in the segment linking $\om$ and $\om_\infty$, respectively.
\begin{thm} (cf. \cite{CFL})
We have
$$
\cs_{\bz}(\gam)=\{(1-\g)\om+\g\om_\infty \mid \gamma\in [0,1]\}\,.
$$
\end{thm}
\noindent
Thus, similar to the Boolean case, the stationary states give rise to the simplest non trivial simplex, with ergodic points given by the vacuum and the state at infinity $\om_\infty$.
However, contrarily to what happens for Booleans, monotone stationary stochastic processes do not satisfy any strong ergodic property like unique ergodicity or unique (weak) mixing, see \cite{CFL}.

\subsection*{Acknowledgements}  The authors kindly acknowledge the support of
Italian INdAM-GNAMPA. They also thank an anonymous referee whose comments nicely improved the presentation of the paper.


\begin{thebibliography}{9999}

\normalsize
\baselineskip=17pt


\bibitem{Acc} L. Accardi,
{\it The noncommutative Markov property}
(Russian), Funkcional. Anal. i Prilo\v{z}en. {\bf 9} (1975), 1-8.

\bibitem{ABCL} L. Accardi, A. Ben Ghorbal, V. Crismale and Y. G. Lu,
{\it Singleton conditions and quantum De Finetti's theorem},
Infin. Dimens. Anal. Quantum Probab. Relat. Top. {\bf 11} (2008),
639-660.

\bibitem{AB} L. Accardi and M. Bo\.{z}ejko
{\it Interacting Fock spaces and Gaussianization of probability measures},
Infin. Dimens. Anal. Quantum Probab. Relat. Top. {\bf 1} (1998),
663-670.

\bibitem{AF1} L. Accardi and F. Fidaleo,
{\it Non homogeneous quantum Markov states and quantum Markov fields},
J. Funct. Anal. {\bf 200} (2003), 324-347.

\bibitem{AFL} L. Accardi, A. Frigerio and J.T Lewis,
{\it Quantum stochastic processes},
Publ. Res. Inst. Math. Sci. {\bf 18} (1982),
97-133.

\bibitem{BCS} T. Banica,  S. R. Curran and R. Speicher,
{\it De Finetti theorems for easy quantum groups}, Ann. Probab. {\bf 40} (2012), no. 1, 401–435.

\bibitem{BF2} S. D. Barreto and F. Fidaleo,
{\it Disordered Fermions on lattices and their spectral properties},
J. Stat. Phys. {\bf 143} (2011), 657-684.

\bibitem{Bill} P. Billingsley, {\it Probability and measure}, John Wiley \& Sons, 1986.

\bibitem{Boz} M. Bo\.{z}ejko,
{\it Uniformly bounded representations of free groups}, J. Reine Angew. Math. {\bf 377} (1987), 170-–186.

\bibitem{Boz2} M. Bo\.{z}ejko,
{\it Deformed Fock spaces, Hecke operators and monotone Fock space of Muraki}, Dem. Math. {\bf XLV} (2012), 399-413.

\bibitem{BSK} M. Bo\.{z}ejko, B. Kummerer and R. Speicher,
{\it q-Gaussian processes: non-commutative and classical aspects},  Commun. Math. Phys. {\bf 185} (1997), 129-154.

\bibitem{BS} M. Bo\.{z}ejko and R. Speicher,
{\it Completely positive maps on Coxeter groups, deformed commutation relations, and operator spaces},  Math. Ann. {\bf 300} (1994), 97-120.

\bibitem{BR1} O. Bratteli and D. W. Robinson, {\it Operator algebras and
quantum statistical mechanics I, II}, Springer, Berlin-Heidelberg-New
York, 1981.

\bibitem{CrF} V. Crismale and F. Fidaleo,
{\it De Finetti theorem on the CAR algebra}, Commun. Math. Phys.
{\bf 315} (2012), 135-152.

\bibitem{CrFid} V. Crismale and F. Fidaleo,
{\it Exchangeable stochastic processes and symmetric states in quantum probability},
Ann. Mat. Pura Appl., {\bf 194} (2015), 969-993.

\bibitem{CFL} V. Crismale, F. Fidaleo and Y. G. Lu,
{\it Ergodic theorems in quantum probability: an application to the monotone stochastic processes}, Ann. Sc. Norm. Sup. Pisa Cl. Sci., to appear, doi: 10.2422/2036-2145.201506\textunderscore{}009, available at arXiv:1505.04688.

\bibitem{CuHu} C. D. Cushen and R. L. Hudson,
{\it  A quantum-mechanical central limit theorem},
J. Appl. Probab. {\bf 8} (1971), 454-469.

\bibitem{F16} F. Fidaleo,
{\it KMS states and the chemical potential for disordered systems},
Commun. Math. Phys. {\bf 262} (2006), 373-391.

\bibitem{F17} F. Fidaleo,
{\it Fermi Markov states}, J. Operator Theory, {\bf 66} (2011), 385-414.

\bibitem{F} F. Fidaleo,
{\it A note on Boolean stochastic processes},
Open Sys. Inform. Dyn. {\bf22} (2015), 1550004, 10 pp.

\bibitem{Ha} U. Haagerup,
{\it An example of non nuclear $C^*$-algebra, which has the metric approximation property}, Invent. Math.
{\bf 50} (1979), 279-293.

\bibitem{HP} R. L. Hudson and K. R. Parthasarathy,
{\it Quantum Ito's formula and stochastic evolutions}, Commun. Math. Phys. {\bf 93} (1984), 301-323

\bibitem{KS} C. K\"{o}stler and R. Speicher,
{\it A noncommutative de Finetti theorem: invariance under quantum permutations is equivalent to freeness with amalgamation}, Commun. Math. Phys. {\bf 291} (2009), 473-490.

\bibitem{Lu} Y. G. Lu,
{\it An interacting free Fock space and the arcsine law}, Prob. Math. Stat. {\bf 17} (1997), 149-166.

\bibitem{O} R. A. Olshen,
{\it The coincidence of measure algebras under an exchangeable probability},
Probab. Theory Relat. Fields {\bf 18} (1971), 153-158.

\bibitem{St2} E. St{\o}rmer,
{\it Symmetric states of infinite tensor products of $C^{*}$-algebras},
J. Funct. Anal. {\bf 3} (1969), 48-68.

\bibitem{Tak} M. Takesaki,
{\it Theory of operator algebras I}, Springer,
Berlin-Heidelberg-New York 1979.

\bibitem{Voi} D. V. Voiculescu,
{\it Symmetries of some reduced free product $C^*$-algebras}, Operator algebras and their connections with topology and ergodic theory, LNM {\bf 1132}, Springer, Berlin-Heidelberg-New
York 1985, 556-588.

\bibitem{VDN} D. V. Voiculescu, K. J. Dykema and A. Nica,
{\it Free random variables}, CRM Monograpy Series, {\bf 1}, American Mathematical Society, Providence, 1992.

\bibitem{WO} N. E. Wegge-Olsen,
{\it K-Theory and $C^*$-algebras}, Oxford University Press,
Oxford-New York-Tokyo 1993.

\bibitem{Wor} S. L. Woronowicz,
{\it Compact matrix pseudogroups}, Commun. Math. Phys. {\bf 111} (1987), 613-665.

\end{thebibliography}
\end{document}